\documentclass[a4paper,12pt,reqno]{amsart}
\usepackage{paralist} 
\usepackage[T1]{fontenc}
\usepackage{geometry}
\allowdisplaybreaks

\usepackage{amsmath}
\allowdisplaybreaks
\usepackage{amssymb}
\usepackage{amsfonts}
\usepackage{color}
\usepackage{enumitem}

\usepackage{booktabs}

\usepackage{etoolbox}
\patchcmd{\thebibliography}{\section*{\refname}}{}{}{}

\usepackage{verbatim}

\usepackage{graphicx}
\usepackage{subfig}
\usepackage{bm}
\usepackage{tikz}
\usepackage{cite}
\bgroup

\numberwithin{equation}{section}

\usepackage[bookmarks,bookmarksopen=false,% Klappt die Bookmarks in Acrobat aus
                  pdfauthor={Autor},%
                  pdftitle={Titel},%
                  colorlinks=true,%
                  linkcolor=blue,%
                  citecolor=blue,%
                  urlcolor=blue,%
                ]{hyperref}

\newtheorem{lemma}{Lemma}[section]
\newtheorem{theorem}[lemma]{Theorem}
\newtheorem{corollary}[lemma]{Corollary}
\newtheorem{proposition}[lemma]{Proposition}

\newtheorem{definition}[lemma]{Definition}

\newtheorem{remark}[lemma]{Remark}

\newcommand{\R}{\mathbb{R}}

\newcommand{\Z}{\mathbb{Z}}

\newcommand{\A}{\mathcal{A}}

\renewcommand{\H}{\mathcal{H}}

\definecolor{darkviolet}{rgb}{0.58,0,0.83} %{148,0,211}

\usetikzlibrary{patterns}

\title{Localised frames for tensor product spaces}
\author[D. Bytchenkoff]{Dimitri Bytchenkoff}
\address{(D. B.) Acoustics Research Institute, Austrian Academy of Sciences, Dominika- nerbastei 16,
1010 Vienna, Austria;
\newline
\indent Faculty of Mathematics, University of Vienna, 
Oskar-Morgenstern-Platz 1, 1090 Vienna, Austria;
\newline 
\indent Université de Lorraine, Laboratoire d’Energétique et de Mécanique Théorique et Appliquée, 2 avenue de la Forêt de
Haye, 54505 Vandoeuvre-lès-Nancy, France}
\email{\hspace{-2pt}dimitri.bytchenkoff@oeaw.ac.at; \hspace{-4pt}dimitri.bytchenkoff@univie.ac.at; dimitri.bytchenkoff@univ-lorraine.fr}
\author[M. Speckbacher]{Michael Speckbacher}
\address{(M. S.) Acoustics Research Institute, Austrian Academy of Sciences, Dominika- nerbastei 16,
1010  Vienna, Austria}
\email{michael.speckbacher@oeaw.ac.at}
\author[P. Balazs]{Peter Balazs}
\address{(P. B.) Acoustics Research Institute, Austrian Academy of Sciences, Dominika- nerbastei 16,
1010  Vienna, Austria;
\newline
\indent Acoustics, Analysis and AI (3AI), Interdisciplinary Transformation University Austria (IT:U), Linz, Austria}
\email{peter.balazs@oeaw.ac.at}
 
\date{}

\begin{document}

\begin{abstract} %A frame is said to be self-localisation of a framemeasures its quality and is defined as the Gram matrixmade from the elements of the frame belonging to a spectral matrix algebra. In this paper we show that, under the conditions that are met by all known examples of self-localised frames, the tensor product of two self-localised frames is a self-localised frame too. To do so we first design an involutive Banach algebra of tensors of rank four made from two solid spectral matrix algebras and prove that this algebra  is indeed inverse-closed. We then show that  the Gram tensor made from the tensor products of the elements of two self-localised frames  belongs to the algebra we designed. Finally we discuss self-localised frames that include operators of rank higher than one for the space of Hilbert-Schmidt operators. 

In this paper, we investigate whether the tensor product of two frames, each individually localised with respect to a spectral matrix algebra, is also localised with respect to a suitably chosen tensor product algebra. % The main challenge lies in finding a suitable matrix algebra. 
We provide a partial answer by constructing an involutive Banach algebra of rank-four tensors that is built from   two  solid spectral matrix algebras. We show that this algebra is inverse-closed, given that the original algebras satisfy a specific property related to operator-valued versions of these algebras. This condition is satisfied by all commonly used solid spectral matrix algebras.
We  then prove that the tensor product of two self-localised frames remains self-localised with respect to our newly constructed tensor algebra. Additionally, we discuss generalisations to localised frames of Hilbert-Schmidt operators, which may not necessarily consist of rank-one operators.
\end{abstract}

\maketitle

\noindent\textit{2020 Mathematics subject classification.} {46B28, 42B35, 42C15}\\
\noindent \textit{Keywords.} {localised frames, tensor products, spectral matrix algebras, inverse-closedness}

\section{Introduction}

\noindent The concept of localisation of frames %for Hilbert spaces
was introduced by %Karlheinz
Gröchenig \cite{Groechenig_2004} to measure, in a broad sense, the quality of a frame. A frame $\Psi=\{\psi_i\}_{i\in I}$ is said to be self-localised   \cite{Balazs_2017,Fornasier_2005} if its Gram matrix $G_{\Psi}  := \left( \langle \psi_{i'} ,  \psi_i \rangle_H \right)_{(i', \, i) \in I^2}$ belongs to an inverse-closed matrix algebra. This definition  implies, among other things,  that the canonical dual of a self-localised frame is also self-localised and allows   to introduce a whole range of Banach spaces, the so-called co-orbit spaces,  which are defined by    decay conditions on the frame coefficients of its elements.
It is common knowledge that the set of elementary tensor products of the elements of two frames %for Hilbert spaces 
%-- also known as the tensor product of two frames --
constitutes a frame \cite{Balazs_2008_bis} -- often called the tensor product frame -- for the tensor product of the Hilbert spaces. 
%, which can be identified with the space of Hilbert-Schmidt operators.
The question that has remained open  is whether the tensor product of two self-localised frames is, in some way, also self-localised, or, in other words, whether there is a natural way  to define an inverse-closed Banach algebra of   tensors of rank four that contains the Gram tensor (Definition~\ref{Gram_4}) of a tensor product frame. Such a construction would allow for a  direct    application of the tools of  co-orbit theory to tensor product frames \cite{Bytchenkoff_2024}.

In this paper, we provide an affirmative answer to this question, under what we believe to be quite reasonable additional conditions. To achieve this, we begin with two solid spectral matrix algebras and construct an algebra of rank-four tensors. We then equip this algebra with a norm derived from the norms of operator-valued variants of the original matrix algebras. 
We demonstrate that the resulting algebra of tensors of rank four is an inverse-closed, albeit non-solid, involutive Banach algebra provided that the operator-valued versions of the matrix algebras used in its construction are inverse-closed. Here we note that solidity of the algebra is not necessary for the intrinsically localised frames to have most of their desirable properties (see Section~\ref{sec:solidity} for a discussion of that matter). It has recently been shown that the most prominent classes of inverse-closed matrix algebras indeed meet  this condition  \cite{Koehldorfer_2024}, which makes our results widely applicable. Finally, we show that the Gram tensor of the tensor product of two self-localised frames belongs to the  algebra  that we constructed.

Our approach enabled us to overcome the obstacles that have prevented previous attempts to apply the methods developed for proving the inverse-closedness of the Jaffard class \cite{Jaffard_1990}, Schur-type algebras \cite{Schur,Sun_2005} and the Sj\"ostrand algebra   \cite{Sjöstrand} to tensor product algebras.
More specifically, those techniques  usually cannot be applied to  tensor products of algebras from different classes or to algebras using anistropic weights. To illustrate this, let us, e.g., consider two  frames $\Psi_1=\{\psi_{1,   k}\}_{k\in\Z}$ and $\Psi_2=\{\psi_{2,   k}\}_{k\in\Z}$ whose Gram matrices belong to the Jaffard class for two distinct weights, i.e., 
\[
\sup_{(k,  l)\in \Z^2}
\big|\langle \psi_{1,   k},\psi_{1,  l}\rangle \big| \nu_{s_1}(k-l)
<\infty,
\quad \text{ and }
\quad
\sup_{(m,  n)\in \Z^2}
\big|\langle \psi_{2,   m},\psi_{2,   n}\rangle \big| \nu_{s_2}(m-n)
<\infty,
\] where $\nu_s(z)=(1+|z|)^{s}$, and $s_1,s_2$ are both greater than one and differ from one another. The Gram tensor of rank four of the tensor product frame $\Psi_1 \otimes \Psi_2$ will then belong to the so-called anisotropic   Jaffard  class, i.e.,
\[
\sup_{(k, l,  m,   n)\in \Z^4}
\big|\langle \psi_{1,  k} \otimes \psi_{2,   m}, \psi_{1,  l} \otimes \psi_{2,   n}\rangle \big| \nu_{s_1}(k-l)\nu_{s_2}
 (m-n)
<\infty.
\] It has remained unclear whether this anisotropic   Jaffard  class is inverse-closed or not. 
One could try to circumvent this obstacle and consider the  higher dimensional Jaffard class with   isotropic weight $\nu_{s}\big((k,n)-(l,m)\big)$. This class is inverse-closed given that   $s>2$, which assumes a degree of  localisation that the frames $\Psi_1$ and $\Psi_2$ might not satisfy though.
In contrast to this example, our approach allows us to deal with tensor products of  two  algebras of the same class with distinct weights and,  of two algebras from distinct classes.

 In order to make this article self-contained, we   recall the definitions and   established facts we drew on in this paper  in Section~\ref{sec:background} before stating and proving our new contributions in Section~\ref{sec:tensor-of-frames}.

\section{Background information and notation}\label{sec:background}

%\noindent In this section, we will review the key concepts and established facts that will be referenced to in this article.

\subsection{Self-localised frames for Hilbert spaces and their associated co-orbit spaces}

\noindent %Both the central subject of this study and the main tool we shall use to deal with it are so-called \emph{frames} for vector spaces. 
The concept of a \emph{frame} generalises that of a basis of a vector space. One of the major advantages of frames over bases is
that a frame can often be designed in a way that allows for a sparser decompositions of the elements of a given vector space than a basis would do.   

\begin{definition} %Let $H$ be a Hilbert space.
A countable set $\Psi : = \{ \psi_i \}_{i \in I}$ of elements of a separable Hilbert space $H$ is called a \emph{frame}  if there exist positive numbers $A_{\Psi}$ and $B_{\Psi}$  such that
\begin{equation}
A_{\Psi} \, \| f \|_H^2 \leqslant
  \sum_{i \in I} \vert \langle f  ,  \psi_i \rangle_H \vert^2 \leqslant
  B_{\Psi} \, \| f \|_H^2
\label{eq:Rahmenungleichung_Hilbert}
\end{equation} for any element $f\in H$.
\label{eq:Rahmen_im_Hilbertraum}
\end{definition}

\noindent The inequalities in (\ref{eq:Rahmenungleichung_Hilbert}), together with the nature of Hilbert spaces, have the following implications. Firstly,  the second inequality in (\ref{eq:Rahmenungleichung_Hilbert}) implies that  the \emph{analysis operator} $C_\Psi$
\begin{equation*}
C_{\Psi} : H \rightarrow  \ell^2(I) ,\quad f \mapsto \{\langle f   ,   \psi_i \rangle_H\}_{i \in I},
\label{eq:Analysenoperator}
\end{equation*} 
 is bounded. Secondly,   the \emph{synthesis operator} $D_\Psi$
\begin{equation*}
D_\Psi : \ell^2(I) \rightarrow H ,\quad \{c_i\}_{i \in I} \mapsto  \sum_{i \in I} c_i \psi_i,
\label{eq:Synthesenoperator}
\end{equation*}
is bounded too. The first inequality  in (\ref{eq:Rahmenungleichung_Hilbert}) implies that there exists at least one other frame $ \Psi^\mathbf{d}  : = \{ \psi_i^\text{d} \}_{i \in I}$ for  $H$, called a \emph{dual frame}, that satisfies
\begin{equation}
f  
   = \sum_{i \in I} \langle f ,   \psi_i \rangle_H
 \,   \psi_i^\text{d}
   = D_{ {\Psi}^\mathbf{d} } C_{\Psi} f 
   = \sum_{i \in I}  \langle f  ,  \psi_i^\text{d}  \rangle_H \, \psi_i 
   = D_{\Psi} C_{ {\Psi^\mathbf{d} }} f,\quad f\in H.
\label{eq:Rahmen_im_Hilbert_Wiederaufbau}
\end{equation}  
Fourthly,  the \emph{frame operator} $S_\Psi$
\begin{equation*}
S_\Psi : H \rightarrow H ,\quad f \mapsto  \sum_{i \in I} \langle f   ,   \psi_i \rangle_H\, \psi_i
\label{eq:Rahmenoperator}
\end{equation*} is bounded, self-adjoint and invertible  for any frame $\Psi  $ \cite{Balazs_2011} and the sequence $\widetilde{\Psi}:=\{ S_{\Psi}^{-1} \psi_i \}_{i \in I}$, known as the \emph{canonical dual frame} \cite{Christensen_2008}, is a dual frame for $\Psi$. Finally, the map $G_{\Psi}$, called the \emph{Gram operator},
\begin{equation*}
\label{eq:Gramscher_Operator}
G_\Psi : \ell^2(I) \rightarrow \ell^2(I) ,\quad \{c_i\}_{i \in I} \mapsto  \sum_{i \in I} \langle \psi_{i'} , \psi_i \rangle_H\, c_i
= \sum_{i \in I} (G_\Psi)_{i', \, i} \, c_i
= C_{\Psi} D_{\Psi} c ,
\end{equation*} where
\begin{equation*}
G_{\Psi} := ((G_{\Psi})_{i', \, i})_{(i', \, i) \in I^2} := \left( \langle \psi_{i'} ,  \psi_i \rangle_H \right)_{(i', \, i) \in I^2}
\label{eq:Gramsche_Matrix}
\end{equation*} is the so-called \emph{Gram matrix}, % made from the elements of the frame $\Psi$, 
 is a bounded operator.

There is also the notion of a frame for Banach spaces \cite{Groechenig_1991}.
\begin{definition}
    Let $X$ be a separable Banach space and $X_d$ a Banach sequence space. A countable set $\Psi : = \{ \psi_i \}_{i \in I}$ of elements of the topological dual $X^*$ of the space $X$ is called an $X_d$-\emph{frame} for the space $X$ if there are positive numbers $A_{\Psi}$ and $B_{\Psi}$ such that
\begin{equation*}
A_{\Psi} \, \| f \|_X \leqslant
  \big\| \{\langle f  ,  \psi_i \rangle_{X, \, X^*}\}_{i\in I} \big\|_{X_d} \leqslant
  B_{\Psi} \, \| f \|_X
\label{eq:Rahmenungleichung_Banach}
\end{equation*} for any element $f\in X$.
\label{eq:Rahmen_im_Banachraum}
\end{definition}

\noindent Despite an apparent similarity of Definitions \ref{eq:Rahmen_im_Hilbertraum} and \ref{eq:Rahmen_im_Banachraum}, there is a very important difference between   frames for Hilbert spaces and those for Banach spaces:
 an $X_d$-frame for a Banach space in itself does not necessarily allow to express   any element of the space in a way similar to that described by (\ref{eq:Rahmen_im_Hilbert_Wiederaufbau}). Therefore, additional assumptions to ensure such a reconstruction  are needed.
 For some Banach function spaces this problem can be solved by using so-called \emph{structured frames} that satisfy only a few specific and easily verifiable criteria \cite{Borup_2007, Nielsen_2012, Nielsen_2014, Voigtlaender_2022, Bytchenkoff_2020, Bytchenkoff_2021, Bytchenkoff_2025}.  Another class of Banach spaces for which this problem can be solved are the so-called  co-orbit spaces \cite{Feichtinger_1988, Feichtinger_1989, Feichtinger_1989_2}, which are the Banach spaces of choice in numerous areas of harmonic analysis. Classically, co-orbit spaces were studied for families generated by integrable group representations. Here, we shall follow the approach introduced by Fornasier and Gr\"ochenig \cite{Fornasier_2005} and consider co-orbit spaces that are generated by \emph{self-localised frames}. In essence, self-localisation of a frame $\Psi$ amounts to the Gram matrix $G_\Psi$ belonging to   a spectral matrix algebra.
\begin{comment}
Let us start with an intuitive notion of localisation \cite{Groechenig_2003} that should serve as  motivation for the general concept of $\mathcal{A}$-localization. Let us assume that we are given a metric $\rho$ on the index set $I$ that satisfies
\begin{equation*}
\inf_{i, \, i' \in I; \, i \neq i'} \rho ( i, i') = C > 0.
\label{eq:getrennte_Index-Menge}
\end{equation*}
In \cite{Groechenig_2003} a frame is called   \emph{self-localised} %set $\Psi = \{ \psi_i \}_{i \in I}$ of a Hilbert space $H$#
if the absolute values of the elements of the   { Gram matrix} $G_\Psi $  decay polynomially as they deviate from the main diagonal of the matrix, i.e., if
\begin{equation}
|(G_\Psi)_{i,i'}\vert=\vert \langle \psi_i , \psi_{i'} \rangle_\H \vert \lesssim (1+ \rho (i, i'))^{-n},\quad i,i'\in I,
\label{eq:Localisation_1}
\end{equation} for  some $n > n_0$.
The matrices whose elements satisfy \eqref{eq:Localisation_1} belong to a  spectral matrix algebra. This allows the generalization of the notion of self-localisation.
\end{comment}

\begin{definition}\label{def:spectral-algebra}
An involutive Banach algebra $\mathcal{A}$ of infinite matrices with norm $\|\cdot\|_{\mathcal{A}}$ is called a \emph{spectral matrix algebra} if
\begin{enumerate}
\item [(i)]
every $A \in \mathcal{A}$ defines a bounded operator on $\ell^2 (I)$, i.e., $\mathcal{A}\subset \mathcal{B}(\ell^2(I))$; 
\item[(ii)] $\mathcal{A}$ is inverse-closed, that is if $A\in \A$ is invertible in $\mathcal{B}(\ell^2(I))$ then $A^{-1}\in\A$ as well;
\end{enumerate}
if, moreover, 
\begin{enumerate}
\item[(iii)] $\mathcal{A}$ is $solid$, i.e., $A \in \mathcal{A}$, together with $\vert  B_{i,j} \vert \leqslant \vert A_{i,j} \vert$ for any $(i, \, j) \in I^2$, implies that $B := \left( B_{i, \, j} \right)_{(i, \, j) \in I^2} \in \mathcal{A}$ and that $\left\| B \right\|_{\mathcal{A}} \leqslant \left\| A \right\|_{\mathcal{A}}$, 
\end{enumerate}
then $\A$ is called a \emph{solid spectral matrix algebra}.
\label{eq:Spektrale-Matrix-Algebra}
\end{definition}
\noindent Examples of solid spectral matrix algebras include, among others, the Jaffard class \cite{Jaffard_1990}, Schur-type algebras \cite{Schur} and the Sj\"ostrand algebra \cite{Sjöstrand}.
\begin{definition}
Let $I  $ be a  countable index set. The set $\Psi = \{ \psi_i \}_{i \in I}$ of elements of $H$ is said to be $\mathcal{A}$-\emph{self-localised}, or  simply \emph{self-localised}, if its Gram matrix $G_\Psi :=(\langle \psi_i , \psi_{i'} \rangle_H)_{(i, \, i') \in I^2}$ belongs to a spectral matrix algebra $\mathcal{A}$. 
\label{eq:Localisation_2_}
\end{definition}

\noindent An $\mathcal{A}$-self-localised frame $\Psi$ for a Hilbert space $H$ allows to generate a whole range of Banach spaces known as \emph{co-orbit spaces} \cite{Feichtinger_1988, Feichtinger_1989, Feichtinger_1989_2}, which we shall introduce after the following   auxiliary definition.

\begin{definition}
Let $\mathcal{A}\subset\mathcal{B}(\ell^2(I))$ be a spectral algebra of infinite matrices. A sequence of positive real numbers $w := \{ w_i \}_{i \in I}\subset \mathbb{R}_{>0}$ is called an $\mathcal{A}$-\emph{admissible} weight if every matrix $A \in \mathcal{A}$ defines a bounded operator on the sequence space $\ell_w^p(I)$  for any $p \in [1, \, \infty ]$.\label{eq:Zugelassenes_Gewicht_}
\end{definition}

\begin{definition}
Let $\mathcal{A}$ be a spectral matrix algebra,   $\Psi : = \{ \psi_i \}_{i \in I}$ an $\mathcal{A}$-localised frame for a Hilbert space $H$, $\widetilde\Psi : = \{ \widetilde\psi_i \}_{i \in I}$ its canonical dual frame, $w := \{ w_i \}_{i \in I}$  an $\mathcal{A}$-admissible weight and
\begin{equation}
H_{00}(\Psi)  
   = \left\lbrace  \sum_{i \in I} c_i \, \psi_i :  \   \{ c_i \}_{i \in I} \in c_{00}  \right\rbrace,
\label{eq:H_00}
\end{equation} 
where $c_{00}$ stands for the space of sequences of complex numbers with only finitely many non-zero terms. The co-orbit space $H_w^p (\Psi),\ 1\leq p<\infty$ is the completion of  $H_{00}(\Psi)$ with respect to the norm $\|f\|_{H^p_w(\Psi)}:=\left\| C_{\widetilde\Psi} f   \right\|_{\ell_w^p}$.  
\label{eq:Coorbit-Raum}
\end{definition}
\noindent A proper definition of the co-orbit space $H_w^\infty (\Psi)$ is technically more involved. As our analysis focuses primarily on the properties of the localisation algebras rather than the co-orbit spaces, we will omit the details here and refer the interested reader to  \cite{Balazs_2017,H-infty} for an in depth discussion.
One of the main advantage  of co-orbit space theory is that the frame for the Hilbert space that was used to generate the range of co-orbit spaces is also a frame for any of the co-orbit spaces   and it allows to decompose and re-synthesise  any of its elements  \cite{Fornasier_2005}.

\subsection{Tensor products of Hilbert spaces and Hilbert-Schmidt operators} 

%Our main aim in this article is to construct a self-localised frame for the space of Hilbert-Schmidt   operators, which would automatically be a frame for the space of bounded linear operators mapping the corresponding co-orbit spaces onto one another \cite{Bytchenkoff_2024}. \textcolor{red}{This is not the main goal. Its not the construction of a localised frame but to construct an inverse-closed algebra. Also the second part is not precise} To do so we shall use the concept of tensor products \cite{Ryan_2002}.

%\begin{definition}
Let $H_1$ and $H_2$ be two Hilbert spaces, $f_1\in H_1$ and $f_2\in H_2$. The \emph{elementary tensor} 
$f_1 \otimes f_2$ of $f_1$ and $f_2$ will be understood as the  rank-one operator  mapping $H_1$ to $H_2$ according to the expression
\begin{equation*}
(f_1  \otimes f_2) (f) := \langle f  ,  f_1 \rangle_{H_1}  f_2,\quad f\in H_1.
\label{eq:Tensorprodukt_als_Operator}
\end{equation*} %that applies to any $f \in H_1$.
It should be pointed out that the elementary tensors are homogeneous in the following sense 
\[
\alpha (f_1 \otimes f_2)=(\overline{\alpha}f_1)\otimes f_2=f_1\otimes (\alpha f_2)  .
\]
%\begin{definition}
The tensor product $H_1\otimes H_2$ %of the Hilbert spaces $H_1$ and $H_2$
is defined as the completion of the linear span of all elementary tensors 
%of the form $f_1  \otimes f_2$ where $f_1 \in H_1$ and $f_2 \in H_2$ respective
with respect to the metric induced by the scalar product 
\begin{equation*}
\langle f_1\otimes f_2 ,  g_1\otimes g_2\rangle_{H_1\otimes H_2} =\overline{\langle f_1  ,  g_1\rangle}_{H_1}\langle f_2  ,  g_2\rangle_{H_2}. 
\end{equation*}
  %  \label{eq:Tensorpordukt_Hilbertraemen}
%\end{definition}
\noindent We note that  the tensor product 
$H_1\otimes H_2$ can be identified with the Banach algebra of \emph{Hilbert-Schmidt operators} $HS(H_1, H_2)$, which is defined as the space of bounded linear operators mapping $H_1$ to $H_2$ equipped with the norm induced by the scalar product
\begin{equation*}
\langle O , O' \rangle_{HS(H_1, \, H_2)} := \sum_{i \in \mathbb{N}} \langle O e_i  ,   O' e_i \rangle_{H_2},
\label{eq:Hilbert_Schmidt-Skalarprodukt}
\end{equation*} 
where $\{ e_i \}_{i \in \mathbb{N}}$ stands for an orthonormal basis of  $H_1$.

\section{Tensor products of spectral matrix algebras}\label{sec:tensor-of-frames}

%\noindent In this section we formulate and prove the new results of this study.

\noindent The set of elementary tensor products $\Psi_1\otimes \Psi_2 := \{\psi_{1,\, i}\otimes\psi_{2,\, j}\}_{(i, \, j)\in I_1 \times I_2}$ of the elements of two frames $\Psi_1:=\{\psi_{1,\, i}\}_{i\in I_1}\subset H_1$ and  $\Psi_2:=\{\psi_{2,\, j}\}_{j\in I_2}\subset H_2$ is known to constitute a frame for the tensor product $H_1 \otimes H_2$. Its canonical dual is given by  $\widetilde{\Psi_1\otimes\Psi_2} := \widetilde{\Psi_1}\otimes \widetilde{\Psi_2} = \big\{\widetilde \psi_{1,\, i} \otimes \widetilde \psi_{2,\, j}\big\}_{(i, \, j)\in I_1 \times I_2}$, see \cite{Balazs_2008_bis}. In what follows we shall refer to $\Psi_1\otimes \Psi_2$ as a  \textit{tensor product frame}. Let us now assume that $\Psi_1$ and $\Psi_2$ are self-localised according to Definition \ref{eq:Localisation_2_}, in other words, let us assume that $G_{\Psi_1} := \left( \langle \psi_{1, \, i'} , \psi_{1, \, i} \rangle \right)_{(i, \, i') \in I^2} \in \mathcal{A}_1$ and $G_{\Psi_2} := \left( \langle \psi_{2, \, j'} ,  \psi_{2, \, j} \rangle \right)_{(j, \, j') \in I^2} \in \mathcal{A}_2$ where $\mathcal{A}_1$ and  $\mathcal{A}_2$ are two spectral matrix algebras and show that $\Psi_1\otimes \Psi_2$ is, in a way, self-localised too. To do so, we shall define a Gram tensor of rank four $G_{\Psi_1 \otimes \Psi_2}$ %made from the elements of the frame $\Psi_1 \otimes \Psi_2$ 
and an   involutive Banach algebra of tensors of rank four and prove that $G_{\Psi_1 \otimes \Psi_2}$ belongs to this algebra. We first give our definition of the Gram tensor $G_{\Psi_1 \otimes \Psi_2}$, which is reminiscent of the ordinary Gram matrix.

\begin{definition}\label{Gram_4} Let $\{ \Omega_{i, \, k} \}_{(i, \, k) \in I_1 \times I_2}\subset H_1\otimes H_2$. The \emph{Gram tensor of rank four} $G_{\Omega}$ %made from the elements $\{ \Omega_{i, \, k} \}_{(i, \, k) \in I_1 \times I_2}\subset H_1\otimes H_2$ %where
%\begin{equation}
  %  \Omega_{i, \, k}
 %   := \sum_{n \in \mathbb{N}} \lambda_{i, \, k, \, n} \psi_{1, \, n} \otimes \psi_{2, \, n}
%\end{equation}
%of the tensor product frame $\Psi_1 \otimes \Psi_2$
is defined by
\begin{equation}
    G_{\Omega} =
   \big(
     \langle \Omega_{i, \, k} , 
             \Omega_{j, \, l} \rangle_{H_1 \otimes H_2}
   \big)_{(i, \, k, \, l, \, j) \in I_1 \times I_2^2  \times I_1}
.
\end{equation}
\end{definition}
\noindent If $\{ \Omega_{i, \, k} \}_{(i, \, k) \in I_1 \times I_2}$  consists of only elementary tensor products of  two families $\Psi_1\subset H_1$ and $\Psi_2\subset H_2$, the Gram tensor reduces to the Kronecker product of the Gram matrices
\begin{equation}
 \begin{split}
  G_{\Psi_1 \otimes \Psi_2} &
  % := \left( (G_{\Psi_1 \otimes \Psi_2 })_{i, \, k, \, l, \, j} \right)_{(i, \, k, \, l, \, j) \in I_1 \times I_2^2 \times I_1} \\
  = \big( \langle \psi_{1, \, j} , \psi_{1, \, i} \rangle_{H_1} \, \langle \psi_{2, \, l} , \psi_{2, \, k} \rangle_{H_2} \big)_{(i, \, k, \, l, \, j) \in I_1 \times I_2^2  \times I_1}\\ &=(G_{\Psi_1})_{(i,j)\in I_1^2}\, (G_{\Psi_2})_{(k,l)\in I_2^2} .
 \end{split}
\label{eq:Gramsche_Hypermatrix}
\end{equation}
We use $\mathcal{T}$ to denote   the set of rank-four tensors whose elements are indexed by $I_1 \times I^2_2 \times I_1$   \cite{Fock_1961}. Clearly $\left( \mathcal{T}, \, \cdot \, , \, + \right)$ -- i.e., the set $\mathcal{T}$, together with   pointwise addition and scalar multiplication -- constitutes a vector space over the field $\mathbb{C}$. To introduce an involutive Banach algebra contained in  $\left( \mathcal{T}, \, \cdot \, , \, + \right)$, we need to introduce  a norm, a multiplication and an involution, for such tensors. %Here are all the necessary definitions. Let us denote the set of all tensors of rank four whose components have indices in $I_1 \times I^2_2 \times I_1$ by $\mathcal{A}$. 

\begin{definition}\label{Norm} Let $\mathcal{A}_1$ and $\mathcal{A}_2$ be involutive Banach algebras of tensors of rank two with norms $\| \cdot  \|_{\mathcal{A}_1}$ and $\|  \cdot  \|_{\mathcal{A}_2}$ whose elements are indexed by $I^2_1$ and $I^2_2$ respectively. For any tensor of rank four $A\in\mathcal{T}$ we define
\begin{equation} \label{eq:us_A1-tilde}
\begin{split}
\left\| A \right\|_{\mathcal{\widetilde A}_1} : =
\left\|
\left(
\left\|
A \big\vert_{\{i\} \times I_2^2 \times \{j\}}
\right\|_{\mathcal{B}(\ell^2(I_2))}
\right)_{(i, \, j) \in I_1^2}
\right\|_{\mathcal{A}_1},
\end{split}
\end{equation} where $A \big\vert_{\{i\} \times I_2^2 \times \{j\}}$ stands for the restriction of  $A$, viewed as a map $I_1 \times I_2^2 \times I_1 \to \mathbb{C}$, to the subset ${\{i\} \times I_2^2 \times \{j\}}$, and
\begin{equation} \label{eq:us_A2-tilde}
\begin{split}
\left\| A \right\|_{\mathcal{\widetilde A}_2} : =
\left\|
\left(
\left\|
A \big\vert_{I_1 \times \{(k, \, l)\} \times I_1}
\right\|_{\mathcal{B}(\ell^2(I_1))}
\right)_{(k, \, l) \in I_2^2}
\right\|_{\mathcal{A}_2},
\end{split}
\end{equation} where $A \big\vert_{I_1 \times \{(k, \, l)\} \times I_1}$ stands for the restriction of  $A$ to the subset $I_1 \times \{(k, \, l)\} \times I_1$. Furthermore, we set $$\widetilde{\A}_1:=\big\{A\in\mathcal{T}:\ \|A\|_{\widetilde{\A}_1}<\infty\big\},\quad \widetilde\A_2:=\big\{A\in\mathcal{T}:\ \|A\|_{\widetilde{\A}_2}<\infty\big\},$$ 
and $\A:=\A_1\cap\A_2$, where $\A$ is equipped with the norm 
\begin{equation} \label{eq:us_grosse_Norm}
\| A \|_{\mathcal{A}} :=
\max \big\{ \left\|
  A 
  \right\|_{\mathcal{\widetilde A}_1}
   ,
  \left\|
  A 
  \right\|_{\mathcal{\widetilde A}_2}
  \big\}  .
\end{equation}
\end{definition}

\noindent Our, perhaps somewhat unusual, way of indexing the elements of tensors of rank four is, we believe, fully justified by the fact that the \emph{double contracted tensor multiplication} \cite{Salencon_1996}, which we shall formally define in a moment, functions very much like an ordinary matrix multiplication. 

\begin{definition}\label{Mul} For any two tensors of rank four $A := (A_{i, \, k, \, l, \, j})_{(i, \, k, \, l, \, j) \in I_1 \times I^2_2 \times I_1}$ and $B := (B_{i, \, k, \, l, \, j})_{(i, \, k, \, l, \, j) \in I_1 \times I^2_2 \times I_1}$, their \emph{doubly contracted tensor product} $A : B$ will be understood as the tensor of rank four defined by
\begin{equation}
\begin{split}
& \left( \left( A : B \right)_{i, \, k, \, l, \, j} \right)_{(i, \, k, \, l, \, j) \in I_1 \times I^2_2 \times I_1} \\
& \hspace{2.0cm} :=
\left( \sum_{n \in I_2} \sum_{m \in I_1} A_{i, \, k, \, n, \, m} \,  B_{m, \, n, \, l, \, j} \right)_{(i, \, k, \, l, \, j) \in I_1 \times I^2_2 \times I_1}
.
\end{split}
\label{eq:Verknüpfung}
\end{equation}
\end{definition}
\noindent Subsequently, we shall prove  that this product is compatible with the norm that defines  $\A$.

\begin{proposition}\label{geil} Let $\mathcal{A}_1$ and $\mathcal{A}_2$ be solid   Banach algebras. Then %  $\| \cdot \|_{\mathcal{A}}$  defines a norm. If, in addition, both $\mathcal{\widetilde A}_1$ and $\mathcal{\widetilde A}_2$
%\begin{equation}\label{Unsere_supergeile_Algebra_1}
%    \left\{ A \in \left( \mathcal{A}, \, \cdot \, , \, + \, , \, : \,  %\right) \, :
%    \| A \|_{\mathcal{\widetilde A}_1} < \infty
%    \, \right\}
%\end{equation} and
%\begin{equation}\label{Unsere_supergeile_Algebra_2}
%    \left\{ A \in \left( \mathcal{A}, \, \cdot \, , \, + \, , \, : \,  %\right) \, :
%    \| A \|_{\mathcal{\widetilde A}_2} < \infty
%    \, \right\}
%\end{equation}
%are inverse-closed, then
\begin{equation}\label{Unsere_supergeile_Algebra}
    \left( \mathcal{A}, \, \cdot \, , \, + \, , \, : \,    ,\,
    \| \cdot \|_{\mathcal{A}}\right)
\end{equation} is a  Banach algebra.
\end{proposition}
\begin{proof} First of all, we are going to show that the mapping $\| \cdot \|_{\mathcal{A}}$ is a norm. Indeed, the function $\| \cdot \|_{\mathcal{\widetilde A}_1}$ is a norm as it is the composition of the operator norm $\|  \cdot  \|_{\mathcal{B}( \ell^2 (I_2))}$ and the norm $\| \cdot \|_{\mathcal{A}_1}$. The map $\| \cdot \|_{\mathcal{\widetilde A}_2}$ is a norm for a similar reason. Finally,  $\| \cdot \|_{\mathcal{A}}$ is a norm as it is a maximum of the norms $\| \cdot \|_{\mathcal{\widetilde A}_1}$ and $\|  \cdot  \|_{\mathcal{\widetilde A}_2}$.

It remains to prove that 
\begin{equation} \label{eq:toto_100}
    \| A : B \|_{\mathcal{A}} \leqslant
    \| A \|_{\mathcal{A}} \| B \|_{\mathcal{A}}
\end{equation} for any $A$ and $B$ in $\A$. %Let $T$ and $B \in \mathcal{A}$. 
According to (\ref{eq:Verknüpfung}), for any given $(i, \, j) \in I_1^2$, the norm of the restriction $\left( A : B \right) \big\vert_{\{i\} \times I_2^2 \times \{j\}}$ of $A : B$  can be expressed as follows
\begin{equation*}
\left( A : B \right) \big\vert_{\{i\} \times I_2^2 \times \{j\}}   =
\sum_{n \in I_2} \sum_{m \in I_1}
A \big\vert_{\{i\} \times I_2 \times \{n\} \times \{m\}}
B \big\vert_{\{m\} \times \{n\} \times I_2 \times \{j\}}
 .
\label{eq:Loft-Satz_4_Webeis_c}
\end{equation*} Using the triangle inequality and the submultiplicativity of the operator norm $\left\| \, \cdot \, \right\|_{\mathcal{B}(\ell^2(I_2))}$ results in
\begin{equation}
\begin{split}
\left\|\left( A : B \right) \big\vert_{\{i\} \times I_2^2 \times \{j\}}\right\|_{\mathcal{B}(\ell^2(I_2))}
&= \left\| \sum_{n \in I_2} \sum_{m \in I_1}
A \big\vert_{\{i\} \times I_2 \times \{n\} \times \{m\}}
B \big\vert_{\{m\} \times \{n\} \times I_2 \times \{j\}}
\right\|_{\mathcal{B}(\ell^2(I_2))} \\ 
&   \leqslant
\sum_{m \in I_1} \left\| \sum_{n \in I_2}
A \big\vert_{\{i\} \times I_2 \times \{n\} \times \{m\}}
B \big\vert_{\{m\} \times \{n\} \times I_2 \times \{j\}}
\right\|_{\mathcal{B}(\ell^2(I_2))}  \\ 
&  \leqslant
\sum_{m \in I_1} \left\| 
A \big\vert_{\{i\} \times I_2^2 \times \{m\}}
\right\|_{\mathcal{B}(\ell^2(I_2))}
\left\| 
B \big\vert_{\{m\} \times I_2^2 \times \{j\}}
\right\|_{\mathcal{B}(\ell^2(I_2))}
 .
\end{split}
\label{eq:Loft-Satz_4_Webeis_d}
\end{equation}
%\textcolor{red}{consistency for indices with or without $\{\}$?!? In def for product its not used, but here all the time}
%\color{blue}{We do not need curly brackets in the definition.} \color{black}
Now $A$ and $B \in \mathcal{ \widetilde A}_1$ as both $A$ and $B \in \mathcal{A}$ and $\mathcal{A} = \mathcal{ \widetilde A}_1 \cap \mathcal{ \widetilde A}_2$. This, in its turn, implies that 
\begin{equation*}
\left(
\left\|
A \big\vert_{\{i\} \times I_2^2 \times \{m\}}
\right\|_{\mathcal{B}(\ell^2(I_2))}
\right)_{(i, \, m) \in I_1^2} \in \mathcal{A}_1,
%\hspace{0.25cm};
\label{eq:Loft-Satz_4_Webeis_a}
\end{equation*} 
and
\begin{equation*}
\left(
\left\|
B \big\vert_{\{m\} \times I_2^2 \times \{j\}}
\right\|_{\mathcal{B}(\ell^2(I_2))}
\right)_{(m, \, j) \in I_1^2} \in \mathcal{A}_1
\, .
\label{eq:Loft-Satz_4_Webeis_b}
\end{equation*} 
Moreover, combining \eqref{eq:Loft-Satz_4_Webeis_d} and the fact that  $\mathcal{A}_1$ is a solid Banach algebra results in
%\begin{equation}
%\begin{split}
%&=\sum_{n \in I_1}
%\left(
%\left\|
%A \big\vert_{\{i\} \times I_2^2 \times \{n\}}
%\right\|_{\mathcal{B}(\ell^2(I_2))}
%\left\|
%B \big\vert_{\{n\} \times I_2^2 \times \{j\}}
%\right\|_{\mathcal{B}(\ell^2(I_2))}
%\right)_{(i, \, j) \in I_1^2} \in \mathcal{A}_1
%\end{split}
%\label{eq:Loft-Satz_4_Webeis_aa}
%\end{equation} and
\begin{equation}
\begin{split}
 \|
A :&\,  B
 \|_{\mathcal{\widetilde A}_1}
  =
\left\|
\left(
\left\|
\left( A : B \right) \big\vert_{\{i\} \times I_2^2 \times \{j\}}
\right\|_{\mathcal{B}(\ell^2(I_2))}
\right)_{(i, \, j) \in I_1^2}
\right\|_{\mathcal{A}_1} 
\\ 
&\leqslant \left\|
\left(
\sum_{m \in I_1} \left\| 
A \big\vert_{\{i\} \times I_2^2 \times \{m\}}
\right\|_{\mathcal{B}(\ell^2(I_2))}
\left\| 
B \big\vert_{\{m\} \times I_2^2 \times \{j\}}
\right\|_{\mathcal{B}(\ell^2(I_2))}
\right)_{(i, \, j) \in I_1^2}
\right\|_{\mathcal{A}_1} \\ 
&   \leqslant
 \left\|\left( \left\| 
A \big\vert_{\{i\} \times I_2^2 \times \{m\}}
\right\|_{\mathcal{B}(\ell^2(I_2))}
\right)_{(i,m)\in I_1^2}\right\|_{\mathcal{A}_1} 
\left\| \left(\left\| 
B \big\vert_{\{n\} \times I_2^2 \times \{j\}}
\right\|_{\mathcal{B}(\ell^2(I_2))}
\right)_{(m,j)\in I_1^2}\right\|_{\mathcal{A}_1} \\ 
&   =
\left\|
A
\right\|_{\mathcal{\widetilde A}_1}
\left\|
B
\right\|_{\mathcal{\widetilde A}_1}
.
\end{split}
\label{eq:Loft-Satz_4_Webeis_f}
\end{equation} 
%Combining (\ref{eq:Loft-Satz_4_Webeis_c}), (\ref{eq:Loft-Satz_4_Webeis_d}) and (\ref{eq:Loft-Satz_4_Webeis_f}) and taking into account the fact that $\mathcal{A}_1$ is a solid algebra results in
%\begin{equation}
%\begin{split}
%\left\|
%A : B
%\right\|_{\mathcal{\widetilde A}_1}
%& =
%\left\|
%\left(
%\left\|
%\left( A : B \right) \big\vert_{\{i\} \times I_2^2 \times \{j\}}
%\right\|_{\mathcal{B}(\ell^2(I_2))}
%\right)_{(i, \, j) \in I_1^2}
%\right\|_{\mathcal{A}_1} \\
%& \leqslant
%\left\| A
%\right\|_{\mathcal{\widetilde A}_1} \left\| B 
%\right\|_{\mathcal{\widetilde A}_1}
%\, .
%\end{split}
%\label{eq:Loft-Satz_4_Webeis_3}
%\end{equation} 
Using similar arguments we infer that
\begin{equation}
\begin{split}
\left\| A : B
\right\|_{\mathcal{\widetilde A}_2}
 \leqslant
\left\| A
\right\|_{\mathcal{\widetilde A}_2} \left\| B 
\right\|_{\mathcal{\widetilde A}_2}
.
\end{split}
\label{eq:Loft-Satz_4_Webeis_4}
\end{equation}
Combining %(\ref{eq:us_grosse_Norm}), 
(\ref{eq:Loft-Satz_4_Webeis_f}) and (\ref{eq:Loft-Satz_4_Webeis_4}) implies (\ref{eq:toto_100}). Indeed
\begin{equation*}
\begin{split}
      \| A : B \|_{\mathcal{A}} &  
    \leqslant \max \{ \left\|
  A : B
  \right\|_{\mathcal{\widetilde A}_1}
   ,
  \left\|
  A : B
  \right\|_{\mathcal{\widetilde A}_2}
  \} \\   &  
    \leqslant \max \{ \left\|
  A
  \right\|_{\mathcal{\widetilde A}_1}
  \left\|
  B
  \right\|_{\mathcal{\widetilde A}_1}
   ,
  \left\|
  A
  \right\|_{\mathcal{\widetilde A}_2}
  \left\|
  B
  \right\|_{\mathcal{\widetilde A}_2}
  \} \\  
  & \leqslant \max \{ \left\|
  A
  \right\|_{\mathcal{\widetilde A}_1}
   ,
  \left\|
  A
  \right\|_{\mathcal{\widetilde A}_2}
  \}
  \max \{ \left\|
  B
  \right\|_{\mathcal{\widetilde A}_1}
   ,
  \left\|
  B
  \right\|_{\mathcal{\widetilde A}_2}
  \} \\
   & = \| A \|_{\mathcal{A}} \| B \|_{\mathcal{A}} \, .
\end{split}
\end{equation*}

%defined by (\ref{Unsere_supergeile_Algebra_1}) and %(\ref{Unsere_supergeile_Algebra_2}).
\end{proof}

Next, we define the adjoint of a tensor of rank four. 

\begin{definition} The adjoint $A^*$ of a tensor of rank four $A\in\mathcal{T}$ will be understood as the tensor of rank four defined by
\begin{equation}
\left( \left( A^* \right)_{i, \, k, \, l, \, j} \right)_{(i, \, k, \, l, \, j) \in I_1 \times I^2_2 \times I_1} :=
\left(\, \overline{A_{j, \, l, \, k, \, i}} \, \right)_{(i, \, k, \, l, \, j) \in I_1 \times I^2_2 \times I_1},
%\hspace{0.25cm};
\label{eq:Involution}
\end{equation} where the bar over $A_{j, \, l, \, k, \, i}$ stands for its complex conjugation. 
\label{eq:Involution_Addition_Multiplikation_Verknüpfung}
\end{definition} 

\noindent This definition of the adjoint is again similar to that of a matrix, where the order of the indices is reversed and the elements of the matrix are complex conjugated.

\begin{proposition}\label{supergeil1}  Let $\mathcal{A}_1$ and $\mathcal{A}_2$ be solid   Banach algebras. Then  
\begin{equation}\label{Unsere_supergeile_Algebra1}
    \left( \mathcal{A}, \, \cdot \, , \, + \, , \, : \,     
    ,\, \ast\, ,\, \| \cdot \|_{\mathcal{A}}\right)
\end{equation} is an involutive  Banach algebra. 
\end{proposition}
\begin{proof}  Clearly, the adjoint of $A\in\mathcal{A}$ satisfies the following properties $\left( A^* \right)^* = A$, $(\alpha A)^\ast=\overline{\alpha} A^\ast$ and $(A+B)^\ast=A^\ast+ B^\ast$ and therefore qualifies as an involution.
Moreover, we observe that, formally,
\begin{equation}
\begin{split}
\left( A : B \right)_{i, \, k, \, l , \, j}^* &
= \sum_{m \in I_1} \sum_{n \in I_2} \overline{ A_{j, \, l, \, n , \, m} }\, \overline{B_{m, \, n, \, k , \, i} }
= \sum_{m \in I_1} \sum_{n \in I_2}   \overline{B_{m, \, n, \, k , \, i}} \,
\overline{ A_{j, \, l, \, n , \, m}} \\
& = \sum_{m \in I_1} \sum_{n \in I_2} (B^*)_{i, \, k, \, n , \, m} \, (A^*)_{m, \, n, \, l , \, j} 
= \left( B^* : A^* \right)_{i, \, k, \, l , \, j}
\end{split}
\label{eq:Involution_Webeis}
\end{equation} 
holds for any $(i, \, k, \, l , \, j) \in I_1 \times I_2^2 \times I_1$. 
%This implies (\ref{eq:Addition_Distributivität_bezüglich_Verknüpfung_2}).
The product on the right hand side is well-defined whenever $A^\ast, B^\ast\in\mathcal{A}$.
To prove that $\A$ is an involutive Banach algebra it thus remains  to show that 
\begin{equation*}
    \| A^* \|_{\mathcal{A}} =
    \| A \|_{\mathcal{A}}
%\hspace{0.25cm}.
\label{eq:Involution_tata}
\end{equation*}
for every $A\in \A.$
We note that the restriction of $A^\ast$ can be understood as the adjoint of a matrix
\begin{equation*}
\left( A^* \right) \big\vert_{\{i\} \times I_2^2 \times \{j\}} 
%=\left( \overline{ A \big\vert_{\{j\} \times I_2^2 \times \{i\}} } \right)^t
= \left( A \big\vert_{\{j\} \times I_2^2 \times \{i\}} \right)^*
 .
\label{eq:tata_1}
\end{equation*} 
%where the symbol $^t$ stands for the matrix transposition. 
This shows that
\begin{equation*}
\left\|
   ( A^\ast) \big\vert_{\{i\} \times I_2^2 \times \{i\}}    
\right\|_{\mathcal{B}(\ell^2(I_2))} =
\left\|
A \big\vert_{\{j\} \times I_2^2 \times \{i\}}
\right\|_{\mathcal{B}(\ell^2(I_2))},
%\hspace{0.25cm};
\label{eq:tata_2}
\end{equation*} 
and so
\begin{equation*}
\begin{split}
\|A^\ast\|_{\widetilde{\A}_1}&= \left\|
\left(
\left\|
A \big\vert_{\{j\} \times I_2^2 \times \{i\}}
\right\|_{\mathcal{B}(\ell^2(I_2))}
\right)_{(i, \, j) \in I_1^2}
\right\|_{\mathcal{A}_1} 
\\
& =
\left\|
\left(
\left(
\left\|
A \big\vert_{\{i\} \times I_2^2 \times \{j\}}
\right\|_{\mathcal{B}(\ell^2(I_2))}
 \right)^*\right)_{(i, \, j) \in I_1^2}
\right\|_{\mathcal{A}_1} 
\\
&   =
\left\|
\left(
\left\|
A \big\vert_{\{i\} \times I_2^2 \times \{j\}}
\right\|_{\mathcal{B}(\ell^2(I_2))}
\right)_{(i, \, j) \in I_1^2}
\right\|_{\mathcal{A}_1} =
\| A \|_{\mathcal{\widetilde A}_1}
\end{split}
%\hspace{0.25cm};
\label{eq:tata_3}
\end{equation*}
for every $A\in\A$. %Combining (\ref{eq:tata_1}), (\ref{eq:tata_2}) and (\ref{eq:tata_3}) results in
%\begin{equation}
 %   \| A^* \|_{\mathcal{\widetilde A}_1} =
 %   \| A \|_{\mathcal{\widetilde A}_1},\quad A\in\A.
%\hspace{0.25cm}.
%\label{eq:Involution_tata_1}
%\end{equation} 
%for any element of the algebra defined by  (\ref{Unsere_supergeile_Algebra}). 
Using similar arguments we conclude that
\begin{equation*}
    \| A^* \|_{\mathcal{\widetilde A}_2} =
    \| A \|_{\mathcal{\widetilde A}_2},\quad A\in\A,
%\hspace{0.25cm}.
\label{eq:Involution_tata_2}
\end{equation*} 
%for any element $A$ of the algebra defined by  (\ref{Unsere_supergeile_Algebra}). 
which in turn implies that $\|A^\ast\|_{\A}=\|A\|_{\A}$, $A\in\A$.
\end{proof}

\noindent With all the preparatory work in place, our main result now follows almost immediately from previous propositions.
\begin{theorem}\label{main}
    Let $\A_1$ and $\A_2$ be two solid spectral algebras such that $\widetilde{\mathcal{A}}_1$ and $\widetilde{\mathcal{A}}_2$ are inverse-closed. Then $\A$ is an involutive inverse-closed Banach algebra.
\end{theorem}
\begin{proof}
The only thing that remains to be proved is that $\A$ is inverse-closed. Let $A\in\A$ be invertible in $\mathcal{B}(\ell^2(I_1\times I_2))$. Then $A^{-1}\in \widetilde{\A}_1$ and $A^{-1}\in\widetilde{\A}_2$ as both   $\mathcal{\widetilde A}_1$ and $\mathcal{\widetilde A}_2$ are  inverse-closed. Therefore, $A^{-1}\in\A$.
\end{proof}

\noindent The assumption that the algebras $\mathcal{\widetilde A}_1$ and $\mathcal{\widetilde A}_2$ are inverse-closed seems quite sensible to us, as all the common inverse closed algebras satisfy this condition. This follows from the results of \cite{Koehldorfer_2024} 
 where operator-valued matrix algebras  were considered. For $I\subset \R^d$, $\vartheta_s(i)=(1+|i|)^s$, and $\vartheta$ any weight function,
 the     Jaffard algebra 
$$
\mathcal{J}_s(I):=\Big\{A:\ \sup_{(i,j)\in I^2}|A_{i,j}|\vartheta_s(i-j)<\infty\Big\},
$$
the weighted Schur-type algebras 
$$
\mathcal{S}^p_\delta(I):=\Big\{A:\  \|A\|_{\mathcal{S}_\delta^p(I)}<\infty\Big\},
$$ 
where 
\begin{align*}
\|A&\|_{\mathcal{S}_\delta^p(I)} :=\hspace{-1pt}\max\hspace{-1pt}\left\{\hspace{-1pt}\sup_{ i \in I}\left(\sum_{j\in I}|A_{i,j}|^p\vartheta_\delta(i-j)^{  p}\right)^{\frac{1}{p}},\sup_{ j \in I}\left(\sum_{i\in I}|A_{i,j}|^p\vartheta_\delta(i-j)^{ p} \right)^{\frac{1}{p}}\hspace{-1pt} \right\},
\end{align*}
and the Sjöstrand algebra 
$$
\mathcal{C}_\vartheta:=\Big\{A:\ \sum_{i\in\Z^d} \sup_{j\in\Z^d}|A_{j,j-i}|\vartheta(i)<\infty\Big\},
$$
are a priori defined as complex-valued infinite matrices. In the following, we will consider  operator-valued versions of the algebras defined above. 
%definitions the absolute value $|A_{i,j}|$ is replaced by $\|A_{i,j}\|_{\mathcal{B}(H)}$, then  we write $\mathcal{OP\big(J}_s(I)\big)$, $\mathcal{OP\big(S}^p_\delta(I)\big)$ and $\mathcal{OP(C}_\vartheta)$ respectively. More generally, 
For a solid matrix algebra $\A$ and a Hilbert space $H$, we define $$\mathcal{OP}(\A):=\big\{(A_{i,j})_{(i,j)\in I^2}:\ A_{i,j}\in \mathcal{B}(H)\ \text{and }(\|A_{i,j}\|_{\mathcal{B}(H)})_{(i,j)\in I^2}\in \A\big\}.$$
If $\A$ is, e.g., the Jaffard class, then this definition reads 
$$
\mathcal{OP}\big(\mathcal{J}_s(I)\big):=\Big\{A_{i,j}\in \mathcal{B}(H):\ \sup_{(i,j)\in I^2}\|A_{i,j}\|_{\mathcal{B}(H)}\vartheta_s(i-j)<\infty\Big\}.
$$
It was   shown in \cite[Theorems~3.10, 4.8 and 6.3]{Koehldorfer_2024} that the  operator-valued versions of the Jaffard class, the weighted Schur-type and the Sj\"ostrand algebra are indeed inverse-closed. 
We make use of these results to show the following.
\begin{corollary}
   Let $I_1,I_2\subset \R^d$ be  relatively separated discrete sets, and  let $\A_n$ be chosen from $\{\mathcal{J}_s(I_n),\mathcal{S}^1_{\delta}(I_n),\mathcal{C}_\nu\}$, $n=1,2$, with $s>d,\delta\in(0,1]$, and $\nu$ a sub-multiplicative and symmetric weight that satisfies $\lim_{n\to\infty}\nu(nz)^{1/n}=1,$ for every $z\in\R^d$. Then $\A$ is an inverse-closed involutive Banach algebra.
\end{corollary}
\begin{proof} 
Let $A\in \A$ be a rank-four tensor. For any given pair of indices $(i,  i')\in I_1^2$, $A_{\{i\}\times I_2^2\times \{i'\}}$   defines a bounded operator on $\ell^2(I_2)$ and it is straightforward to show that   $\mathcal{OP}(\A_1)=\widetilde{\A}_1$. Consequently, $\widetilde{\A}_1$ is inverse-closed according to the results of \cite{Koehldorfer_2024}. A similar argument holds for $\widetilde{\A}_2$. Applying Theorem~\ref{main} then concludes the proof. 
\end{proof}

\begin{remark}
    It is not known to us whether the operator-valued analog of a spectral matrix algebra is necessarily inverse closed. An affirmative answer to this question would allow us to show that the tensor algebra $\A$ is inverse-closed as long as $\A_1$ and $\A_2$ are inverse-closed.
\end{remark}

\section{Solidity in co-orbit theory}\label{sec:solidity}

\noindent While solidity of a spectral matrix algebra $\A$ is one of the standard assumptions in co-orbit theory \cite{Fornasier_2005,Balazs_2017}, the tensor product algebra that we designed and use in this study is  not solid in general. 
We are, however, convinced that solidity is not essential to prove the main statements of co-orbit theory. Instead, it is the other assumptions outlined in Definitions~\ref{def:spectral-algebra} and \ref{eq:Zugelassenes_Gewicht_} that play the central roles. While a detailed discussion confirming this claim would go beyond the scope of this paper, we briefly discuss how the properties of the Gram matrix are used to derive the results of co-orbit theory.
A recurring argument is to
rewrite certain operations as the multiplication of the (cross-)Gram matrix    by an  $\ell^p_w(I)$-sequence and then use that elements of the spectral algebra define bounded
operators on each $\ell^p_w(I)$, $p \in [1,  \infty ]$, as long as $w$ is an $\A$-admissible weight (see Definition \ref{eq:Zugelassenes_Gewicht_}). This, for example, allows one to show boundedness of the synthesis operator or that the co-orbit spaces are independent of the particular  self-localised frame that is used to define them. Inverse-closedness on the other hand  is used to prove the fundamental result of co-orbit theory that the canonical dual of a self-localised frame is also self-localised.

To the best of our knowledge, there are only two occasions where solidity is indeed used in the literature on   co-orbit theory, namely in \cite[Theorem~3]{Balazs_2017}, where the argument could be readily  adapted to work in the operator-valued setting that we consider in this paper, and in \cite[Definition~4 \& Lemma~2.2]{Fornasier_2005}, where the somewhat different notion of $\sharp$-self-localisation is introduced and used to describe the mutual localisation of two frames with indices in two distinct non-uniform index sets.

\section{Localised frames for the tensor product \texorpdfstring{$H_1\otimes H_2$}{H1xH2}}

\noindent The following result is an almost immediate consequence of   Definition \ref{Norm} being applied to the tensor product of two self-localised frames.

\begin{proposition}\label{mega}
    Let $\Psi_1$ be an $\mathcal{A}_1$-self-localised frame for   $H_1$ and $\Psi_2$ an $\mathcal{A}_2$-self-localised frame for  $H_2$ where $\mathcal{A}_1$ and $\mathcal{A}_2$ are two spectral matrix algebras such that $\mathcal{\widetilde A}_1$ and $\mathcal{\widetilde A}_2$ are inverse-closed. Then the tensor product frame $\Psi_1 \otimes \Psi_2$ is an $\mathcal{A}$-self-localised frame for   $H_1 \otimes H_2$. % where the algebra $\mathcal{A}$ is defined by the expression (\ref{Unsere_supergeile_Algebra}).
\end{proposition}
\begin{proof} According to (\ref{eq:Gramsche_Hypermatrix}) and  (\ref{eq:us_A1-tilde}),
\begin{equation*} \label{eq:us_A1-tilde_G}
\begin{split}
\left\| G_{\Psi_1 \otimes \Psi_2} \right\|_{\mathcal{\widetilde A}_1} & =
\left\|
\left(
\left\|
G_{\Psi_1 \otimes \Psi_2} \big\vert_{\{i\} \times I_2^2 \times \{j\}}
\right\|_{\mathcal{B}(\ell^2(I_2))}
\right)_{(i, \, j) \in I_1^2}
\right\|_{\mathcal{A}_1} \\
& =
\left\|
\left\|
G_{\Psi_2}
\right\|_{\mathcal{B}(\ell^2(I_2))}
\left((
G_{\Psi_1})_{i,\, j}
\right)_{(i, \, j) \in I_1^2}
\right\|_{\mathcal{A}_1} \\
& = \left\|
G_{\Psi_2}
\right\|_{\mathcal{B}(\ell^2(I_2))}
\left\|
G_{\Psi_1}
\right\|_{\mathcal{A}_1}
\end{split}
\end{equation*}
and, similarly,
\begin{equation*} \label{eq:us_A1-tilde_G_bis}
\left\| G_{\Psi_1 \otimes \Psi_2} \right\|_{\mathcal{\widetilde A}_2} 
= \left\|
G_{\Psi_1}
\right\|_{\mathcal{B}(\ell^2(I_1))}
\left\|
G_{\Psi_2}
\right\|_{\mathcal{A}_2} . 
\end{equation*}
Therefore, both $\left\| G_{\Psi_1 \otimes \Psi_2} \right\|_{\mathcal{\widetilde A}_1}$ and $\left\| G_{\Psi_1 \otimes \Psi_2} \right\|_{\mathcal{\widetilde A}_2}$ are finite as $G_{\Psi_1} \in \mathcal{A}_1\subset \mathcal{B}(\ell^2(I_1))$ and  $G_{\Psi_2} \in \mathcal{A}_2\subset \mathcal{B}(\ell^2(I_2))$. From this we infer that  $\left\| G_{\Psi_1 \otimes \Psi_2} \right\|_{\mathcal{A}} <\infty $ and so $G_{\Psi_1 \otimes \Psi_2} \in \mathcal{A}$.
\end{proof}

\noindent It should be pointed out here that   self-localisation with respect to $\A$ need not be restricted to sets of rank-one operators. Frames of Hilbert-Schmidt operators have been  recently studied, for example, in \cite{Henry_1} in the context of quantum harmonic analysis \cite{Skrettingland}. In that study, the authors developed a theory of operator-valued modulation spaces by analysing operator families of the form 
 $\{\pi(z_i)S\pi(w_k)^\ast\}_{(i,k)\in I^2}$, where $\pi(z)$ represents a time-frequency shift by $z\in \R^{2d}$ (see, e.g., \cite{Groechenig_2001}) and $S\in HS(L^2(\R^d),L^2(\R^d))$. They explored the localisation of operator-valued frames through the lens of classical co-orbit theory for integrable group representations.  To derive their results they made  extensive use of the connection between  operator-valued frames and  modulation spaces in double phase space. 
Our approach, on the other hand, is much more general and applicable to situations where no structure is available or where it is natural to combine different types of localisation.

%\textcolor{red}{Frames of higher rank operators instead of rank one operators (=elementary tensors). Make connection to Franz's group and QHA}  
%\cite{Henry_1,Henry_2}
%polarized Cohen's class
%quantum time-frequency analysis
%discretization and modulation spaces
%\cite{Skrettingland}

\section*{Acknowledgements}

\noindent 
This research was funded by the Austrian Science Fund (FWF)  10.55776/P34624 (P.B. and D.B.) and 10.55776/PAT1384824 (M.S.).  
For open
access purposes, the authors have applied a CC BY public copyright license to any author-accepted manuscript
version arising from this submission.

\end{document}